\documentclass[11pt]{article}
\usepackage{amssymb,amsmath,latexsym}
\usepackage{eufrak}
\usepackage[all]{xy}
\usepackage{amsthm}
\author{Eric Emtander\footnote{Department of Mathematics,
Stockholm University, 106 91 Stockholm, erice@math.su.se}}
\title{A class of hypergraphs that generalizes chordal graphs}
\begin{document}
\date{}
\maketitle \theoremstyle{plain}\newtheorem{thm}{Theorem}[section]
\theoremstyle{plain}\newtheorem{lemma}{Lemma}[section]
\theoremstyle{plain}\newtheorem{prop}{Proposition}[section]
\theoremstyle{plain}\newtheorem{cor}{Corollary}[section]
\theoremstyle{definition}\newtheorem{de}{Definition}[section]
\theoremstyle{remark}\newtheorem{rem}{Remark}[section]
\def\mapr#1{{\mathop{\rightarrow}\limits^{#1}}}
\let\sse=\subseteq \let\ssm=\smallsetminus\let\hra=\hookrightarrow
\let\car=\curvearrowright\let\ss=\subset\let\t=\textrm\let\mc=\mathcal
\abstract{In this paper we introduce a class of hypergraphs that we
call chordal. We also extend the definition of triangulated
hypergraphs, given in \cite{VT}, so that a triangulated hypergraph,
according to our definition, is a natural generalization of a
chordal (rigid circuit) graph. In \cite{F1}, Fr\"oberg shows that
the chordal graphs corresponds to graph algebras, $R/I(\mc{G})$,
with linear resolutions. We extend Fr\"oberg's method and show that
the hypergraph algebras of generalized chordal hypergraphs, a class 
of hypergraphs that includes the chordal hypergraphs, have linear
resolutions. The definitions we give, yield a natural higher 
dimensional version of the well known flag property of simplicial 
complexes. We obtain what we call $d$-flag complexes.}

\section{Introduction and preliminaries}

Let ${\mc X}$ be a finite set and ${\mc E}=\{E_1,\ldots,E_s\}$ a
finite collection of non empty subsets of ${\mc X}$. The pair ${\mc{
H=(X,E)}}$ is called a {\bf hypergraph}. The elements of ${\mc X}$
and ${\mc E}$, respectively, are called the {\bf vertices} and the
{\bf edges}, respectively, of the hypergraph. If we want to specify
what hypergraph we consider, we may write $\mc{X(H)}$ and
$\mc{E(H)}$ for the vertices and edges respectively. A hypergraph is
called {\bf simple} if: (1) $|E_i|\ge2$ for all $i=1,\ldots,s$ and
(2) $E_j\sse E_i$ implies $i=j$. If the cardinality of ${\mc X}$ is
$n$ we often just use the set $[n]=\{1,2,\ldots,n\}$ instead of
$\mc{X}$.

Let $\mc H$ be a hypergraph. A {\bf subhypergraph} $\mc K$ of $\mc
H$ is a hypergraph such that $\mc{X}(\mc{K})\sse\mc{X(H)}$, and
$\mc{E(K)}\sse\mc{E(H)}$. If $\mc Y\sse\mc{X}$, the {\bf induced
hypergraph on} $\mc Y$, $\mc{H}_{\mc{Y}}$, is the subhypergraph with
$\mc{X(H_Y)}=\mc{Y}$ and with $\mc{E(H_Y)}$ consisting of the edges
of $\mc{H}$ that lie entirely in $\mc{Y}$. A hypergraph $\mc H$ is
said to be {\bf $d$-uniform} if $|E_i|=d$ for every edge $E_i\in
\mc{E( H)}$. Note that a simple $2$-uniform hypergraph is just an
ordinary simple
graph.\\

Throughout the paper we denote by $R$ the polynomial ring
$k[x_1,\ldots,x_n]$ over some field $k$, where $n$ is the number of
vertices of a hypergraph considered at the moment. By identifying
each vertex $v_i\in\mc{X(H)}$ with a variable $x_i\in R$, we may
think of an edge $E_i$ of a hypergraph as a monomial
$x^{E_i}=\prod_{j\in E_i}x_j$ in $R$. Employing this idea, we may
associate to every simple hypergraph $\mc{H}$, a squarefree monomial
ideal in $R$. The {\bf edge ideal} $I(\mc H)$ of a hypergraph $\mc
H$ is the ideal
\[
I(\mc{H})=(x^{E_i}\,;\, E_i\in\mc{E(H)})\sse R,
\]
generated ``by the edges'' of ${\mc H}$. This yields the
{\bf hypergraph algebra} $R/I(\mc{H})$.\\

In this way we obtain a 1-1 correspondence
\[
\big\{\t{simple\,hypergraphs\,on}\,[n]\big\}\leftrightsquigarrow
\big\{\t{squarefree\,monomial\,ideals}\,I\sse
R=k[x_1,...,x_n]\big\}.
\]

Recall that an {\bf (abstract) simplicial complex} on vertex set
$[n]$ is a collection, $\Delta$, of subsets of $[n]$ with the
property that $G\sse F,\,F\in\Delta \Rightarrow G\in\Delta$. The
elements of $\Delta$ are called the {\bf faces} of the complex and
the maximal (under inclusion) faces are called {\bf facets}. The
{\bf dimension}, $\dim F$, of a face $F$ in $\Delta$, is defined to
be $|F|-1$, and the dimension of $\Delta$ is defined as
$\dim\Delta=\max\{\dim F ;\, F\in \Delta\}$. Note that the empty set
$\emptyset$ is the unique $-1$ dimensional face of every complex
that is not the void complex $\{\}$ which has no faces. The
dimension of the void complex may be defined as $-\infty$. The
$r$-skeleton of a simplicial complex $\Delta$, is the collection of
faces of $\Delta$ of dimension at most $r$. Let $V\sse [n]$. We
denote by $\Delta_V$ the simplicial complex
\[
\Delta_V=\{F\sse [n]\,;\,F\in\Delta, F\sse V\}.
\]
For convenience, we consider 0 to be a natural number, i.e.,
$\mathbb{N}=\{0,1,2,3,\ldots\}$. A vector ${\bf
j}=(j_1,\ldots,j_n)\in\{0,1\}^n$ is called a squarefree vector in
$\mathbb{N}^n$. We may identify $\bf{j}$ with the set $V\sse [n]$,
where $i\in V$ precisely when $j_i=1$. Since this correspondence
between the $V$ and the
$\bf{j}$ is bijective, we may also denote $\Delta_V$ by 
$\Delta_{{\bf j}}$.\\

Given a simplicial complex $\Delta$, we denote by $\mc{C}.(\Delta)$
its reduced chain complex, and by
{$\tilde{H}_n(\Delta;k)=Z_{n}(\Delta)/B_n(\Delta)$} its $n$'th
reduced homology group with coefficients in the field $k$. In
general we could use an arbitrary abelian group instead of $k$, but
we will only consider the case when the coefficients lie in a field.
For convenience, we define the homology of the void
complex to be zero.\\

Let $\Delta$ be an arbitrary simplicial complex on $[n]$. The {\bf
Alexander dual simplicial complex} $\Delta^\ast$ to $\Delta$, is
defined by
\[
\Delta^\ast=\{F\sse [n] ; [n]\ssm F \not\in\Delta\}.
\]
Note that $(\Delta^\ast)^\ast=\Delta$.\\

The edge ideal was first introduced by R.~H.~Villarreal \cite{Vi},
in the case when $\mc{H}=\mc{G}$ is a simple graph. After that,
hypergraph algebras has been widely studied. See for instance
\cite{E1, Fa1, VT, VT3, VT2, He1, Ja, Vi2, Zh}. In \cite{VT}, the
authors use certain connectedness properties to determine a class of
hypergraphs such that the hypergraph algebras have linear
resolutions. Furthermore, nice recursive formulas for computing the
Betti numbers are given.

Perhaps the most common way to study the connections between the
combinatorial information contained in a hypergraph, and the
algebraic information contained in the corresponding hypergraph
algebra, is the one given by the {\it Stanley-Reisner
correspondence}, which is a 1-1 correspondence:
\[
\big\{\t{simplicial\,complexes\,on}\,[n]\big\}\leftrightsquigarrow
\big\{\t{squarefree\,monomial\,ideals}\,I\sse
R=k[x_1,\ldots,x_n]\big\}
\]
\[
\Delta\leftrightsquigarrow I_\Delta.
\]
Here, a monomial $x^F$ is an element in $I_\Delta$ precisely when
$F$ is a non face in $\Delta$. Note that using the above two 1-1
correspondences, we also get a 1-1 correspondence between the class
of simple hypergraphs on $[n]$, and the class
of simplicial complexes on $[n]$.\\

Let ${\mc H}=([n],\mc{E(H)})$ be a simple hypergraph and consider
its edge ideal $I(\mc H)\sse R$. Note that $R/I(\mc H)$ is precisely
the Stanley-Reisner ring of the simplicial complex
\[
\Delta(\mc H)=\{F\sse [n] ; E\not\sse F,\, \forall E\in \mc{E(H)}\}.
\]
This is called the {\bf independence complex} of $\mc H$. The edges
in $\mc H$ are precisely the minimal non faces of $\Delta(\mc H)$.

Thus, we may think of the edges of a simple hypergraph as the
minimal non faces of a simplicial complex or, equally well, the
relations in the $k$-algebra $R/I(\mc{H})$. The connections between
a (hyper)graph and its independence complex are explored in, for
example \cite{E1, F1, Ja}.

Another way to use hypergraphs to investigate the properties of
simplicial complexes was introduced in \cite{Fa1} by S.~Faridi.
Given a simplicial complex $\Delta$, denote by $\{F_1,\ldots,F_t\}$
the set of facets of $\Delta$. Faridi then defines another
squarefree monomial ideal, the {\bf facet ideal of $\Delta$},
\[
\mc{F}(\Delta)=(x^F\,;\,F\,\t{is\,a\,facet\,of}\,\Delta).
\]
In several papers, for example \cite{Fa1, Fa2, Zh}, properties of
simplicial complexes are studied via the combinatorial properties of
their facet ideals. Note that the
set of facets of $\Delta$ is a simple hypergraph.\\

In section 4, we introduce the classes of chordal and triangulated
hypergraphs. The definition of triangulated hypergraph is almost
identical to Definition 5.5 in \cite{VT}, however, ours is more
general. These classes of hypergraphs illustrates that $d$-uniform 
hypergraphs behaves much like ordinary simple graphs. However, there 
are familiar properties of graphs that do not translate immediately 
to $d$-uniform hypergraphs. See for instance Remark 4.1 and Example 1.

It is well known, see \cite{FG}, that chordal graphs are
characterized by the fact that they have perfect elimination orders.
We show that this remain true for hypergraphs.

In Theorem 4.1 we show that the properties of being triangulated,
chordal, and having
a perfect elimination order, are equivalent also for hypergraphs.\\

In section 5 we introduce the class of generalized chordal
hypergraphs, which includes the chordal hypergraphs, and show that
the corresponding hypergraph algebras, $R/I(\mc{H})$, have linear
resolutions. Our method of proof is a natural generalization of one
used by R.~Fr\"oberg in \cite{F1}. There, Fr\"oberg characterizes,
in terms of the complementary graphs $\mc{G}^c$, precisely for what
graphs $\mc{G}$ the graph algebras $R/I(\mc{G})$ have linear
resolutions. Fr\"oberg shows:

\begin{thm} Let $G$ be a simple graph on $n$ vertices. Then
$k[x_1,\ldots,x_n]/I(\mc{G})$ has linear resolution precisely when
$\mc{G}^c$ is chordal (rigid circuit, triangulated,\ldots).
\end{thm}

By Theorem 5.1, we obtain a partial generalization of Fr\"oberg's 
theorem.\\

Let $\Delta$ be an arbitrary simplicial complex, such that the
Stanley-Reisner ring $R/I_\Delta$ has linear resolution. Then we
know that the generators of $I_\Delta$ all have the same degree, $d$
say. Thus, we may think of $R/I_\Delta$ as a hypergraph algebra
$R/I(\mc{H})$ for some $d$-uniform hypergraph $\mc{H}$. However, we
will look at things in another way. The {\bf complementary
hypergraph $\mc{H}^c$}, of a $d$-uniform hypergraph $\mc{H}$, is
defined as the hypergraph on the same set of vertices as $\mc{H}$,
and edge set
\[
\mc{E(H}^c)=\{F\sse\mc{X(H)}\,;\,|F|=d,\,F\not\in\mc{E(H)}\}.
\]
The edges of $\mc{H}^c$ may, in a natural way, be thought of as the
$(d-1)$-dimensional faces in the independence complex
$\Delta(\mc{H})$, of $\mc{H}$. This is how Fr\"oberg looks at things
when he proves his theorem. We show that the complex
$\Delta(\mc{H})$ is completely determied by the edges in $\mc{H}^c$,
which gives us the notion of $d$-flag complexes.

\section{Resolutions and Betti numbers}

To every finitely generated graded module $M$ over the polynomial
ring $R=k[x_1,\ldots,x_n]$, we may associate a {\bf minimal
($\mathbb{N}$-)graded free resolution}
\[
0\to {\bigoplus}_j R(-j)^{\beta_{l,j}(M)}\to{\bigoplus}_j
R(-j)^{\beta_{{l-1},j}(M)} \to\cdots\to{\bigoplus}_j
R(-j)^{\beta_{0,j}(M)}\to M\to 0
\]
where $l\le n$ and $R(-j)$ is the $R$-module obtained by shifting
the degrees of $R$ by $j$. Thus, $R(-j)$ is the graded $R$-module in
which
the grade $i$ component $(R(-j))_i$ is $R_{i-j}$.\\
The natural number $\beta_{i,j}(M)$ is called the $ij$'th
$\mathbb{N}$-{\bf graded Betti number} of $M$. If $M$ is multigraded
we may equally well consider the $\mathbb{N}^n$-graded minimal free
resolution and Betti numbers of $M$. The difference lies just in the
fact that we now use multigraded shifts $R(-\bf j)$ instead of
$\mathbb{N}$-graded ones. The {\bf total} $i$'th Betti number is
$\beta_i(M)=\sum_j \beta_{i,j}$. For further details on resolutions,
graded rings and Betti numbers, we
refer the reader to \cite{BH}, sections 1.3 and 1.5.\\
\\
The Betti numbers of $M$ occur as the dimensions of certain vector
spaces over $k=R/m$, where $m$ is the unique maximal graded ideal in
$R$. Accordingly, the Betti numbers in general depend on the
characteristic of $k$.\\
A minimal free resolution of $M$ is said to be {\bf linear} if for
$i>0$, $\beta_{i,j}(M)=0$ whenever $j\neq i+d-1$ for some fixed
natural number $d\ge1$.

In connection to this we mention the {\it Eagon-Reiner theorem}.

\begin{thm} Let $\Delta$ be a simplicial complex and $\Delta^\ast$ its
Alexander dual complex. Then $R/I_\Delta$ is Cohen-Macaulay if and
only if $R/I_{\Delta^\ast}$ has linear minimal free resolution.
\end{thm}
\begin{proof}
See \cite{ER}, Theorem 3.
\end{proof}

\section{Hochster's formula and the Mayer-Vietoris sequence}

In topology one defines Betti numbers in a somewhat different
manner. {\it Hochster's formula} provides a link between these and
the Betti numbers defined above.

\begin{thm}{(Hochster's formula).} Let $R/I_\Delta$ be the Stanley-Reisner
ring of a simplicial complex $\Delta$. The non-zero Betti numbers of
$R/I_\Delta$ are only in squarefree degrees $\bf j$ and may be
expressed as
\[
\beta_{i,{\bf j}}(R/I_\Delta)=\dim_k\tilde{H}_{|{\bf
j}|-i-1}(\Delta_{\bf j};k).
\]
Hence the total $i$'th Betti number may be expressed as
\[
\beta_i(R/I_\Delta)=\sum_{V\sse[n]}\dim\tilde{H}_{|V|-i-1}(\Delta_V;k).
\]
\end{thm}
\begin{proof}
See \cite{BH}, Theorem 5.5.1.
\end{proof}
If one has $\mathbb{N}^n$-graded Betti numbers, it is easy to obtain
the $\mathbb{N}$-graded ones via
\[
\beta_{i,j}(R/I_\Delta)=\sum_{\substack{ {\bf j'} \in {\mathbb{N}^n} \\
|{\bf j'}|=j}} \beta_{i,{\bf j'}}(R/I_\Delta).
\]
Thus,
\[
\beta_{i,j}(R/I_\Delta)=\sum_{\substack{V\sse [n]\\
|V|=j}}\dim\tilde{H}_{|V|-i-1}(\Delta_V;k).
\]

Recall that if we have an exact sequence of complexes,\footnote{That
is, complexes of modules over some ring $R$.}
\[
{\bf 0}\to{\bf L}\to{\bf M}\to{\bf N}\to{\bf 0}
\]
there is a long exact (reduced) homology sequence associated to it
\[
\cdots\to H_r(N)\to H_{r-1}(L)\to H_{r-1}(M)\to H_{r-1}(N)\to\cdots
.
\]
When we prove Theorem 5.1, we will use this homology sequence in the
special case
where it is associated to a simplicial complex as follows.\\

Suppose we have a simplicial complex $N$ and two subcomplexes $L$
and $M$, such that $N=L\cup M$. This gives us an exact sequence of
(reduced) chain complexes
\[
0\to\mc{C}.(L\cap M)\to\mc{C}.(L)\oplus\mc{C}.(M)\to\mc{C}.(N)\to0.
\]
The non trivial maps here are defined by $x\mapsto (x,-x)$ and
$(x,y)\mapsto x+y$.

The long exact (reduced) homology sequence associated to this
particular sequence, is called the Mayer-Vietoris sequence. More
about the Mayer-Vietoris sequence can be found in \cite{Ma}, section
4.4.

\section{The classes of chordal and triangulated hypergraphs}
In this section, all hypergraphs are assumed to be simple and
$d$-uniform.

\begin{de}
Two distinct vertices $x,y$ of a hypergraph $\mc{H}$ are {\bf
neighbours} if there is an edge $E\in\mc{E(H)}$, such that $x,y\in
E$. For any vertex $x\in\mc{X(H)}$, the {\bf neighbourhood of} $x$,
denoted $N(x)$, is the set
\[
N(x)=\{y\in\mc{X(H)}\,;\,y\,\mathrm{is\,a\,neighbour\,of}\,x\}.
\]
If $N(x)=\emptyset$, $x$ is called isolated. Furthermore, we let
$N[x]=N(x)\cup\{x\}$ denote the {\bf closed neighbourhood of} $x$.
\end{de}

\begin{rem}
Let $\mc{H}$ be a hypergraph and $V\sse\mc{X(H)}$. Denote by $N_V[x]$ 
the closed neigbourhood of $x$ in the induced hypergraph $\mc{H}_V$. 
For ordinary graphs it is clear that $N_V[x]=N[x]\cap V$. This is not 
always the case for hypergraphs, as is shown in the example below. 
Note that the notation $N_V[x]$ will only occur in this remark 
and the example below. The fact that we do not make any greater use of it, 
is intimately connected to, and in a sense illustrates, the properties of 
the hypergraphs that we are to consider.   
\end{rem}

{\bf Example 1.}\quad Consider the hypergraph $\mc{H}$ on vertex set 
$\mc{X(H)}=\{a,b,c,d,e\}$ and edge set 
$\mc{E(H)}=\{\{a,b,c\},\{a,d,e\}, \{b,c,d\}\}$. Let $V=\{a,b,c,d\}$. Then 
$N_V[a]=\{a,b,c\}$ but $N[a]\cap V=\{a,b,c,d\}$.\\

Recall the definition of the $d$-complete hypergraph:
\begin{de}
The $d$-complete hypergraph, $K_n^d$, on a set of $n$ vertices, is
defined by
\[
\mc{E}(K_n^d)={[n]\choose d}
\]
where $F\choose d$ denotes the set of all subsets of $F$, of
cardinality $d$. If $n<d$, we interpret $K_n^d$ as $n$ isolated
points.
\end{de}

If $\mc{H}$ is a hypergraph, we associate a simplicial complex
$\Delta_{\mc{H}}$ to it in the following way:

\begin{de} Given a hypergraph $\mc{H}=(\mc{X(H)},\mc{E(H)})$, the {\bf
complex of} $\mc{H}$, $\Delta_{\mc{H}}$, is the simplicial complex
\[
\Delta_{\mc{H}}=\{F\sse\mc{X(H)}\,;\,{F\choose d}\sse\mc{E(H)}\}
\]
Note that this implies that if $F\sse\mc{X(H)}$, $|F|<d$, then
$F\in\Delta_{\mc{H}}$.
\end{de}

\begin{rem}
$\Delta_{\mc{H}}$ is completely determined by $\mc{H}$. Indeed, all
faces of dimension at least $d-1$ clearly is determined by $\mc{H}$,
since each one correspond uniquely to a $d$-complete subhypergraph
of $\mc{H}$.
\end{rem}
\begin{rem}
Recall that a flag complex is a simplicial complex in which every
minimal non face consists of precisely 2 elements. As one easily
sees, such complex is determined by its 1-skeleton. According to the
previous remark, {\bf $d$-flag complexes}, i.e. complexes
$\Delta_{\mc{H}}$ whose minimal non faces all have cardinality $d$,
in a natural way generalizes flag complexes.
\end{rem}

\begin{prop}
$\Delta_{\mc{H}}=\Delta(\mc{H}^c)$, where $\Delta(\mc{H}^c)$ is the
independence complex of $\mc{H}^c$.
\end{prop}
\begin{proof}
The two complexes has the same set of vertices.
$F\in\Delta(\mc{H}^c)$ precisely when ${F\choose d}\sse\mc{E(H)}$.
Furthermore, $F\in\Delta(\mc{H}^c)$ for every $F\sse\mc{X(H)}$ with
$|F|<d$.
\end{proof}

\begin{de}
Let $\Delta$ be a simplicial complex on a finite set, $\mc{X}$, of
vertices. For any given $d\in\mathbb{N}$, the {\bf $d$-uniform
hypergraph}, $\mc{H}_d(\Delta)$, {\bf of} $\Delta$, is the
hypergraph with vertex set $\mc{X}$, and with edge set
\[
\mc{E}_d(\Delta)=\{F\in\Delta\,;\,|F|=d\}.
\]
\end{de}

\begin{prop}
Let $\mc{H}$ be a hypergraph and $\Delta$ an arbitrary $d$-flag
complex on $\mc{X(H)}$. Then,
\begin{itemize}
\item{$\mc{H}_d(\Delta_{\mc{H}})=\mc{H}$}
\item{$\Delta_{\mc{H}_d(\Delta)}=\Delta$}
\end{itemize}
\end{prop}
\begin{proof}
This follows directly from Definition 4.2 and Definition 4.3.
\end{proof}

\begin{de}
A hypergraph $\mc{H}$ is called {\bf triangulated} if for every non
empty subset $V\sse\mc{X(H)}$, either there exists a vertex $x\in V$
such that the induced hypergraph $\mc{H}_{N[x]\cap V}$ is isomorphic
to a $d$-complete hypergraph $K_n^d$, $n\ge d$, or else the edge set
of $\mc{H}_V$ is empty.
\end{de}

\begin{de}
A hypergraph $\mc{H}$ is called {\bf triangulated}* if for every non
empty subset $V\sse\mc{X(H)}$, either there exists a vertex $x\in V$
such that $N[x]\cap V$ is a facet of $(\Delta_{\mc{H}})_V$ of
dimension greater than or equal to $d-1$, or else the edge set of
$\mc{H}_V$ is empty.
\end{de}

\begin{de}
A {\bf chordal hypergraph} is a $d$-uniform hypergraph, obtained
inductively as follows:
\begin{itemize}
\item{$K_n^d$ is a chordal hypergraph, $n,d\in\mathbb{N}$.}
\item{If $\mc{G}$ is chordal, then so is
$\mc{H}$=$\mc{G}\bigcup_{K_j^d} K_i^d$, for $0\le j<i$. (This we
think of as glueing $K_i^d$ to $\mc{G}$ by identifying some edges,
or parts of some edges, of $K_i^d$ with the corresponding part,
$K_j^d$, of $\mc{G}$.)}
\end{itemize}
\end{de}

\begin{rem}
For $d=2$ this specializes precisely to the class of generalized
trees, i.e. generalized $n$-trees for some $n$, as defined in
\cite{F1}.
\end{rem}

\begin{rem}
In the special case of simple graphs, Definition 4.5 specializes
precisely to the ordinary chordal (rigid cicuit) graphs. Recall that
a simple graph is called chordal if every induced cycle of length
$>3$, has a chord. By considering minimal cycles, it is clear that a
graph that is triangulated according to Definition 4.5, is chordal.
Assume a graph $\mc{G}$ is chordal. It follows from Theorems 1 and 2
in \cite{Di}, that the chordal graphs are precisely the generalized
trees (see Remark 4.4). In a generalized tree we may easily find a
vertex $x$, with the property that $\mc{G}_{N[x]}$ is complete, as
follows: We know that $\mc{G}=\mc{G}'\cup_{K_j}K_i$, $0\le j< i$.
Then, we just pick a vertex $x\in\mc{X}(K_i)\ssm\mc{X(G')}$, since
such $x$ clearly has the property that $\mc{G}_{N[x]}$ is complete.
Since every induced subgraph of a chordal graph is chordal, the same
thing holds for every $\mc{G}_V$, $V\sse\mc{X(G)}$.
\end{rem}

Another characterization of chordal graphs may be found in
\cite{FG}. There it is shown that a simple graph is chordal
precisely when it has a perfect elimination order. Recall that a
perfect elimination order of a graph $\mc{G}=(\mc{X},\mc{E})$ is an
ordering of its vertices, $x_1<x_2<\cdots<x_n$, such that for each
$i$, $\mc{G}_{N[x_i]\cap\{x_{i},x_{i+1},\ldots,x_n\}}$ is a complete
graph. The concept of perfect elimination order is well suited for
generalizations. We make the following

\begin{de}
A hypergraph $\mc{H}$ is said to have a {\bf perfect elimination
order} if its vertices can be ordered $x_1<x_2<\cdots<x_n$, such
that for each $i$, either
$\mc{H}_{N[x_i]\cap\{x_{i},x_{i+1},\ldots,x_n\}}$ is isomorphic to a
$d$-complete hypergraph $K_n^d$, $n\ge d$, or else $x_i$ is isolated
in $\mc{H}_{\{x_{i},x_{i+1},\ldots,x_n\}}$
\end{de}

Note that this specializes precisely to the definition of perfect
elimination order for simple graphs if we put $d=2$.

\begin{lemma}
Let $\mc{H}$ be a hypergraph and $x\in V\sse\mc{X(H)}$ a vertex such 
that $\mc{H}_{N [x]}\cong K_m^d$, $m\ge d$. Then $\mc{H}_{N[x]\cap V}$ 
either is isomorphic to a $d$-complete hypergraph $K_{m'}^d$, $m'\ge d$, 
or else $x$ is isolated in $V$.
\end{lemma}
\begin{proof}
Either $|N[x]\cap V|\ge d$ or else $|N[x]\cap V|<d$.
\end{proof}
 
\begin{rem}
The above lemma in some sense explains what goes on in the 
proofs hereafter. It also casts some light on the last comment made 
in Remark 4.1.
\end{rem}

\begin{lemma}
If a hypergraph $\mc{H}$ with $\mc{E(H)}\neq\emptyset$ has a perfect
elimination order, then it has a perfect elimination order
$x_1<x_2<\cdots<x_n$ in which $x_1$ is not isolated.
\end{lemma}

\begin{proof}
Let $x_1<x_2<\cdots<x_n$ be a perfect elimination order of $\mc{H}$,
and put
\[
t=\min\{i\,;\,x_i\,\mathrm{is\,not\,isolated}\}.
\]
We claim that $x_t<\cdots<x_n<x_1<\cdots<x_{t-1}$ also is a perfect
elimination order of $\mc{H}$. Since $x_1,\ldots,x_{t-1}$ are
isolated, we need only verify that
$\mc{H}_{N[x_i]\cap\{x_i,x_{i+1},\ldots,x_n,x_1,\ldots,x_{t-1}\}}\cong
K_{m_i}^d$ for some $m_i\ge d$, $i=t,\ldots,n$. However, this is
clear since
$\mc{H}_{N[x_i]\cap\{x_i,x_{i+1},\ldots,x_n,x_1,\ldots,x_{t-1}\}}=
\mc{H}_{N[x_i]\cap\{x_i,x_{i+1},\ldots,x_n\}}$.

\end{proof}

\begin{lemma}
If a hypergraph $\mc{H}$ is triangulated (triangulated*, chordal),
or, has a perfect elimination order, then so does $\mc{H}_V$ for
every $V\sse\mc{X(H)}$.
\end{lemma}
\begin{proof}
Let $V\sse\mc{X(H)}$. If $\mc{E(H}_V)=\emptyset$, $\mc{H}_V$ clearly is 
triangulated and triangulated*. It is also chordal since
we can add one vertex at a time until we have the desired discrete
hypergraph, and any ordering of $V$ yields a perfect elimination
order. Thus we may assume that $\mc{E(H}_V)\neq\emptyset$.

The lemma is clear for the classes of triangulated and triangulated*
hypergraphs, since if $W\sse V$, we have that
$(\mc{H}_V)_W=\mc{H}_W$. Now, let $\mc{H}=\mc{G}\bigcup_{K_j^d}
K_i^d$, $0\le j<i$, be chordal. If $V\sse\mc{X(G)}$, or if 
$V\sse\mc{X}(K_i^d)$, we are done by induction. If this is not the 
case, it is easy to realize that 
$\mc{H}_V=\mc{G}_V\bigcup_{(K_j^d)_V}(K_i^d)_V$. Since $\mc{G}_V$ is
chordal by induction, the result follows.  Finally, assume $\mc{H}$
has a perfect elimination order $x_1<x_2<\cdots<x_n$. Then $V$ inherits 
an ordering $x_{i_1}<x_{i_2}<\cdots<x_{i_{|V|}}$. The fact that this is a
perfect elimination order of $\mc{H}_V$ follows from Lemma 4.1.

\end{proof}

\begin{thm}
Let $\mc{H}=(\mc{X(H)},\mc{E(H)})$ be a $d$-uniform hypergraph. Then
the following are equivalent.
\begin{itemize}
\item[$(i)$]{$\mc{H}$ is triangulated.}
\item[$(ii)$]{$\mc{H}$ is triangulated*.}
\item[$(iii)$]{$\mc{H}$ is chordal.}
\item[$(iv)$]{$\mc{H}$ has a perfect elimination order.}
\end{itemize}
\end{thm}

\begin{proof}
Due to Lemma 4.3, we need only consider the full set $\mc{X(H)}$ of
vertices
in our arguments, and we may assume that $\mc{E(H)}\neq\emptyset$.\\

$(i)\Rightarrow (ii)$.\quad Since we assume $\mc{E(H)}\neq\emptyset$
and consider only the case where $V=\mc{X(H)}$, there is a vertex
$x$ such that $\mc{H}_{N[x]}\cong K_n^d$, $n\ge d$. Then, $N[x]$
clearly is a face in $\Delta_{\mc{H}}$ of dimension at least $d-1$.
Furthermore it has to be a facet, since if there were a $y\in
\mc{X(H)}$, $y\neq x$, such that $N[x]\cup\{y\}\in \Delta_{\mc{H}}$,
then there would exist an edge $E$ with
$x,y\in E$. Hence, $y\in N[x]$.\\

$(ii)\Rightarrow (i)$.\quad By assumption, there is a vertex $x$ such 
that $N[x]$ is a facet in $\Delta_{\mc{H}}$ of dimension greater 
than or equal to $d-1$, whence it is clear (from the definition of 
$\Delta_{\mc{H}})$ that $\mc{H}_{N[x]}\cong K_n^d$ for some $n\ge d$.\\

$(i)\Rightarrow (iii)$.\quad  By assumption there is a vertex
$x\in\mc{X(H)}$ such that $\mc{H}_{N[x]}\cong K_n^d$, for some $n\ge
d$. Let $\mc{G}$ be the induced hypergraph on $\mc{X(H)}\ssm\{x\}$.
Then $\mc{E(G)}$ consists of all edges of $\mc{H}$, except those
that contain $x$. This yields $\mc{H}=\mc{G}\cup_\mc{K}K_n^d$, where
$\mc{K}=K_{|N(x)|}^d$ on vertex set $N(x)$, and by induction we are
done.\\

$(iii)\Rightarrow (i)$.\quad Assume
$\mc{H}=\mc{G}\cup_{K_j^d}K_i^d$, $0\le j<i$, is chordal, where
$\mc{G}$ is chordal by construction. If $i\ge d$, any vertex
$x\in\mc{X}(K_i^d)\ssm\mc{X(G)}$ will do, since $\mc{H}_{N[x]}\cong
K_i^d$ for such $x$. If $i<d$, we find, by induction, a vertex
$x\in\mc{X(G)}$ with the property that
$\mc{H}_{N[x]}=\mc{G}_{N[x]}\cong K_n^d$ for some $n\ge d$, since
otherwise the edge set of $\mc{H}$ would be empty, contrary to our
assumptions.\\

$(i)\Rightarrow (iv)$.\quad By assumption we find a vertex $x=x_1$
such that $\mc{H}_{N[x_1]}\cong K_n^d$, $n\ge d$. Since the induced
hypergraph on $\mc{X(H)}\ssm\{x_1\}$ is triangulated, by induction
it has a perfect elimination order $x_2<\cdots<x_n$. If we put
$x_1<x_2$ we are done.\\

$(iv)\Rightarrow (i)$.\quad By Lemma 4.2 there is a perfect
elimination order $x_1<\cdots<x_n$, such that $\mc{H}_{N[x_1]\cap
V}\cong K_m^d$ for some $m\ge d$.

\end{proof}

\subsection{Some examples}
In \cite{E1}, we considered hypergraph generalizations of the well
known complete and complete multipartite graphs. We use these to
create some examples of
chordal hypergraphs.\\

Recall from \cite{E1} the definition of the $d$-complete bipartite
hypergraph $K_{n,m}^d$: This is the hypergraph on a vertex set that
is a disjoint union, $[n]\sqcup [m]$, of two finite sets. The edge
set consists of all sets $V\sse [n]\sqcup [m]$, $|V|=d$, such that 
$V\cap [n]\neq\emptyset\neq V\cap [m]$.\\

{\bf Example 2.}\quad Here we consider the complement
$\mc{H}=(K_{n,m}^d)^c$ of $K_{n,m}^d$. We claim that $\mc{H}$ is
chordal. It is easy to see, considering the relations in the
Stanley-Reisner ring, that $\Delta_{\mc{H}}$ looks like
\[
(\Delta_n\sqcup\Delta_m)\cup\Gamma_{d-2}([n]\cup [m])
\]
where $\Delta_r$ is the full simplex on $[r]$, and
$\Gamma_{d-2}([n]\cup [m])$ is the
$(d-2)$-skeleton of the full simplex on $[n]\sqcup [m]$.\\
Clearly, the $d$-uniform hypergraph of this complex, in other words
$\mc{H}$, is the disjoint union two $d$-complete hypergraphs,
\[
\mc{H}=K_n^d\cup_{K_0^d}K_m^d,
\]
so $\mc{H}$ is chordal.\\

{\bf Example 3.}\quad Now consider the complex $\Delta_{K_{n,m}^d}$,
of $K_{n,m}^d$. If $n,m<d$, we have an isomorphism $K_{n,m}^d\cong
K_{n+m}^d$, so in this case $K_{n,m}^d$ is chordal. If $n$ or $m$ is
greater than or equal to $d$, $K_{n,m}^d$ is not chordal. This is
because no matter which vertex $x$ we choose, the induced hypergraph
on $N[x]$ cannot be $d$-complete, since it would then contain
an edge lying entirely in either $[n]$ or $[m]$, which is impossible.\\
The general case of the $d$-complete multipartite hypergraph,
$K_{n_1,\ldots,n_t}^d$, is similar. $K_{n_1,\ldots,n_t}^d$ is
chordal only when $n_i<d$ for every $i=1,\ldots,t$.
The arguments are the same as in the bipartite case.\\

Another kind of complete hypergraph, is the $d(a,b)$-complete
hypergraph $\mc{H}=K_{n,m}^{d(a,b)}$, where $d=a+b$, $a,b\ge1$. Here
$\mc{X(H)}=[n]\sqcup [m]$,
and $\mc{E(H)}={[n]\choose a}\times {[m]\choose b}$.\\

{\bf Example 4.}\quad Consider the complex of $K_{n,m}^{d(a,b)}$.
Pick any vertex $x$ and consider $N[x]$. If the induced hypergraph
$(K_{n,m}^{d(a,b)})_{N[x]}$ is to be complete, both $n$ and $m$ must
be smaller than $d$, and at least one of the two equations $n=a$,
$m=b$ must hold. Otherwise we obtain a contradiction since
$K_{n,m}^{d(a,b)}$ would then contain an edge of the wrong shape. If
$n$ and $m$ satisfy these conditions, the hypergraph is chordal.

\section{Generalized chordal hypergraphs}
It is easy to find an example of a $d$-uniform hypergraph $\mc{H}$
that is not chordal, but such that the Stanley-Reisner ring of
$\Delta_{\mc{H}}$ has linear
resolution.\\

{\bf Example 5:}\, Let $\mc{H}$ be the 3-uniform hypergraph with
$\mc{X(H)}=\{a,b,c,d\}$, and $\mc{E(H)}=\big\{ \{a,b,c\}, \{a,c,d\},
\{a,b,d\}\big\}$. The following simple picture lets us visualize
$\mc{H}$.
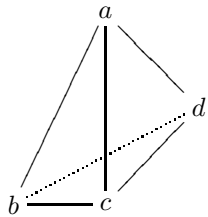
\begin{displaymath}
\xymatrix{&a\ar@{-}[ddl]\ar@{-}[dd]\ar@{-}[dr]&\\
&&d\ar@{-}[dl]\\
b\ar@{-}[r]\ar@{.}[urr]&c&}
\end{displaymath}

$R/I_{\Delta_{\mc{H}}}$ has linear resolution, but $\mc{H}$ is not chordal.\\

If $\Delta$ is a simplicial complex on $[n]$ and $E$ is a finite
set, we denote by $\Delta\cup E$ the simplicial complex on $[n]\cup
E$ whose set of facets, $\mc{F}(\Delta\cup E)$, is
$\mc{F}(\Delta)\cup \{E\}$. Similarly, if $\mc{H}$ is a (not
necessarily $d$-uniform) hypergraph and $E$ a finite set, we denote
by $\mc{H}\cup E$ the hypergraph on $\mc{X(H)}\cup E$ whose edge set
is $\mc{E(H}\cup E)=\mc{E(H)}\cup \{E\}$.

\begin{de}
A {\bf generalized chordal hypergraph} is a $d$-uniform hypergraph,
obtained inductively as follows:
\begin{itemize}
\item{$K_n^d$ is a generalized chordal hypergraph, $n,d\in\mathbb{N}$.}
\item{If $\mc{G}$ is generalized chordal, then so is
$\mc{H}$=$\mc{G}\bigcup_{K_j^d} K_i^d$, for $0\le j<i$.}
\item{If $\mc{G}$ is generalized chordal and $E\sse\mc{X(G)}$ a finite set,
$|E|=d$, such that at least one element of ${E\choose {d-1}}$ is not
a subset of any edge of $\mc{G}$, then $\mc{G}\cup E$ is generalized
chordal.}
\end{itemize}
\end{de}

\begin{rem}
It is clear that every chordal hypergraph is also a generalized
chordal hypergraph. Furthermore, for $d=2$ chordal graphs and
generalized chordal graphs are the same.
\end{rem}

\begin{thm}
Let $\mc{H}=(\mc{X(H)},\mc{E(H)})$ be a generalized chordal
hypergraph and $k$ a field of arbitrary characteristic. Then the
Stanley-Reisner ring of $\Delta_{\mc{H}}$ has linear resolution.
\end{thm}

\begin{proof}
We consider the three instances of Definition 5.1 one at a time. If
$\mc{H}\cong K_n^d$ we are done, since if $n\ge d$ we have a simplex
so the situation is trivial, and if $n<d$ the claim is proved for
example in \cite{E1}, Theorem 3.1. So, we may assume
$\mc{H}\not\cong K_n^d$. Let $\mc{H}=\mc{G}\cup_{K_j^d}K_i^d$, $0\le
j<i$, where $\mc{G}$ is generalized chordal. Let $C$ and $B$ be the
simplices determined by $K_j^d$ and $K_i^d$, respectively, and
consider the complex $\Delta'_{\mc{H}}=\Delta_{\mc{G}}\bigcup B$.
Note that $B\cap\Delta_{\mc{G}}=C$, $B\neq C$. We first show that
$\Delta'_{\mc{H}}$ has linear resolution. For every
$V\sse\mc{X(H)}$, we have an exact sequence of chain complexes
\[
0\to\mc{C}.(C_V)\to\mc{C}.((\Delta_{\mc{G}})_V)\oplus\mc{C}.(B_V)\to
\mc{C}.((\Delta'_{\mc{H}})_V)\to0.
\]
By induction, via Hochster's formula, we know that
$(\Delta_{\mc{G}})_V$ can have non zero homology only in degree
$d-2$. But then, since both $B_V$ and $C_V$ are simplices and
accordingly have no homology at all, by considering the
Mayer-Vietoris sequence we conclude that the only possible non zero
homologies of $(\Delta'_{\mc{H}})_V$ lies
in degree $d-2$.\\
Note that it is not in general true that
$\Delta_{\mc{H}}=\Delta'_{\mc{H}}$. In fact, this holds only when
$d=2$. However, the difference between the two complexes is easy to
understand, and we may use the somewhat easier looking
$\Delta'_{\mc{H}}$ to show that $\Delta_{\mc{H}}$ has linear
resolution as well.

To this end, let $\Gamma_{d-2}(\mc{X(H)})$ be the $(d-2)$-skeleton
of the full simplex on vertex set $\mc{X(H)}$. Then one sees that
\[
\Delta_{\mc{H}}=\Delta'_{\mc{H}}\cup\Gamma_{d-2}(\mc{X(H)}).
\]
The $(d-2)$-faces that we add to $\Delta'_{\mc{H}}$ to obtain
$\Delta_{\mc{H}}$, can certainly not cause any homology in degrees
greater than $d-2$, that did not already exist in
$\Delta'_{\mc{H}}$. Indeed, suppose $\sum_i a_i\sigma_i$ is a cycle
in a degree $r>d-2$, where $a_i\in k$ and the $\sigma_i$'s are faces
of $\Delta_{\mc{H}}$, of dimension $r$. Since every face $\sigma_i$
actually lies in $\Delta'_{\mc{H}}$, it follows that $\sum_i
a_i\sigma_i$ is a cycle also in $\Delta'_{\mc{H}}$. Thus, if
$\Delta'_{\mc{H}}$ has linear resolution, so does $\Delta_{\mc{H}}$.

Finally, let $\mc{H}=\mc{G}\cup E$. Let $F_1,\ldots,F_t$ be the
elements of ${E\choose {d-1}}$ that are not subsets of any edge of
$\mc{G}$. Note that $\Delta_{\mc{H}}=\Delta_{\mc{G}}\cup E$. Take 
$V\sse\mc{X(H)}$. If $E\not\sse V$, then $(\Delta_{\mc{H}})_V=
(\Delta_{\mc{G}})_V$, so, by induction we conclude that the only 
possible non zero homologies of $(\Delta_{\mc{H}})_V$ lies in 
degree $d-2$. Hence we may assume that $E\sse V$. Then we have 
an exact sequence
\[
0\to\mc{C}.((\Delta_{\mc{G}}\cap
E)_V)\to\mc{C}.((\Delta_{\mc{G}})_V)\oplus\mc{C}.(E_V)
\to\mc{C}.((\Delta_{\mc{H}})_V)\to0.
\]

Note that $E_V$ is a simplex so it has no homology, and, by
induction, we know that $R/I_{\Delta_{\mc{G}}}$ has linear
resolution. Using Hochster's formula, we may conclude that
$\tilde{H}_{d-1}((\Delta_{\mc{G}})_V;k)=0$. Hence, the
Mayer-Vietoris sequence obtained from the above exact sequence looks
as follows:
\[
0\to\tilde{H}_{d-1}((\Delta_{\mc{H}})_V)\to
\tilde{H}_{d-2}((\Delta_{\mc{G}}\cap E)_V)\to
\tilde{H}_{d-2}((\Delta_{\mc{G}})_V)\to
\tilde{H}_{d-2}((\Delta_{\mc{H}})_V)\to0.
\]

Let $z=\sum_j a_j\sigma_j$ be an element in
$Z_{d-1}((\Delta_{\mc{H}})_V)$, where $\sigma_1=E$. Consider the
expression for the derivative of this cycle
\[
0=d(z)=\cdots+\sum_{i=1}^t\pm a_1F_i+\cdots.
\]

Since $\sum_{i=1}^t\pm a_1F_i$ only can come from $d(E)$, we
conclude that $a_1=0$. Hence $z\in Z_{d-1}((\Delta_{\mc{G}})_V)$,
and, using Hochster's formula, we may conclude that the
Stanley-Reisner ring of $\Delta_{\mc{H}}$ has linear resolution.

\end{proof}

\begin{cor}
Let $\mc{H}=(\mc{X(H)},\mc{E(H)})$ be a generalized chordal
hypergraph and $k$ a field of arbitrary characteristic. Then the
Stanley-Reisner ring $R/I_{\Delta_{\mc{H}}^*}$ of the Alexander dual
complex $\Delta_{\mc{H}}^*$ is Cohen-Macaulay.
\end{cor}
\begin{proof}
This follows by the Eagon-Reiner theorem.
\end{proof}

\begin{cor}
Theorem 5.1 and Corollary 5.1 in particular applies to triangulated
and triangulated* hypergraphs, and also to hypergraphs that have
perfect elimination orders.
\end{cor}

{\bf Question 1:}\quad Is there a hypergraph $\mc{H}$ such that the
Stanley-Reisner ring of $\Delta_{\mc{H}}$ has linear resolution over
any field $k$, but that is not a generalized chordal hypergraph?\\

{\bf Question 2:}\quad If $\mc{H}$ is a generalized chordal
hypergraph, are there more equivalent characterizations of $\mc{H}$
similar to those for a chordal hypergraph given in Theorem 4.1?

\bibliographystyle{plain}
\bibliography{ref}

\end{document}